\documentclass{article}
\usepackage{amsmath,graphicx}

\usepackage{amsthm}
\usepackage{amssymb}
\usepackage{xypic}
\usepackage{fancyhdr}
\usepackage{multirow}
\usepackage{hyperref}
\usepackage{algpseudocode}
\usepackage{algorithm}

\theoremstyle{plain}
\newtheorem{theorem}{Theorem}
\newtheorem{proposition}{Proposition}

\newtheorem{lemma}{Lemma}

\theoremstyle{definition}
\newtheorem{definition}{Definition}

\def\rank{\textrm{rank }}

\def\id{\textrm{id }}

\fancypagestyle{plain}{
\fancyhf{}
\lhead{}
\rhead{}
\chead{}
\lfoot{\tiny{This material is based upon work supported by the Defense Advanced Research Projects Agency (DARPA) under Contract No. HR001124C0319. Any opinions, findings and conclusions or recommendations expressed in this material are those of the author(s) and do not necessarily reflect the views of the Defense Advanced Research Projects Agency (DARPA). Distribution Statement ``A'' (Approved for Public Release, Distribution Unlimited).}}
}
\title{Probing the topology of the space of tokens with structured prompts}
\author{Michael Robinson$^1$, Sourya Dey$^2$, Taisa Kushner$^3$}
\date{$^1$Mathematics and Statistics, American University, Washington, DC, USA, michaelr@american.edu\\%
  $^2$Galois, Inc., Arlington, VA, USA, sourya@galois.com\\%
  $^3$Galois, Inc., taisa@galois.com
}

\begin{document}

\maketitle

\begin{abstract}
This article presents a general and flexible method for prompting a large language model (LLM)  to reveal its (hidden) token input embedding up to homeomorphism.  Moreover, this article provides strong theoretical justification---a mathematical proof for generic LLMs---for why this method should be expected to work.
With this method in hand, we demonstrate its effectiveness by recovering the token subspace of {\tt Llemma-7B}.  The results of this paper apply not only to LLMs but also to general nonlinear autoregressive processes.
\end{abstract}

\pagestyle{plain}

\section{Introduction}
\label{sec:introduction}

The set of tokens $T$, when embedded within the latent space $X$ of a large language model (LLM) can be thought of as a finite sample drawn from a distribution supported on a topological subspace of $X$.
One can ask what the smallest (in the sense of inclusion) subspace and simplest (in terms of fewest free parameters) distribution can account for such a sample.

Previous work \cite{robinson2024structure} suggests that the smallest topological subspace from which tokens can be drawn is not manifold, but has structure consistent with a stratified manifold.
That paper relied upon knowing the \emph{token input embedding function} $T \to X$,
which given each token $t \in T$, ascribes a representation in $X$.
Because embeddings preserve topological structure, in this paper, we will study $T$ by equating it with the image of the token input embedding function,
thereby treating $T$ both as the set of tokens and as a subspace of $X$.
This subspace is called the \emph{token subspace} of $X$.
Usually $X$ is taken to be Euclidean space $\mathbb{R}^n$,
so the token input embedding function is stored as a matrix,
whose rows correspond to tokens and whose columns correspond to coordinates within the latent space.
For instance, for the LLM {\tt Llemma-7B}, there are 32016 tokens that are embedded in $\mathbb{R}^{4096}$,
so consequently the token input embedding is a $32016 \times 4096$ matrix.

The token subspace has striking topological and geometric structure.  For instance, it has a definite local dimension near each token, different clusters of semantically related tokens, and curvature within subspaces that admit such a definition. 
It is natural to ask whether this structure has a measurable impact on the ``behavior'' of the LLM, namely its response to queries.
Moreover, a significant limitation of \cite{robinson2024structure} is that it relies upon direct knowledge of the token subspace via the token input embedding function.
The authors only considered open source models because the token input embedding function is distributed as part of the model.
This article shows that this limitation can be lifted, and in so doing also shows that the topology of the token subspace has a direct and measurable impact on LLM behavior.
Specifically, we show that an unknown token subspace can be recovered up to \emph{homeomorphism} (a kind of strong topological equivalence) by way of structured prompting of the LLM, without further access to its internal representations.

\subsection{Contributions}

This article presents a general and flexible method (Algorithm \ref{alg:token_prompting}) for prompting an LLM to reveal its (hidden) token subspace up to homeomorphism, and provides strong theoretical justification (Theorem \ref{thm:autoregressive_embedding}) for why this method should be expected to work.
With this method in hand, we demonstrate its effectiveness in Section \ref{sec:results} by recovering the token subspace of {\tt Llemma-7B}.
Recognizing that LLMs are a kind of generalized autoregressive process,
the proof of Theorem \ref{thm:autoregressive_embedding} applies to general \emph{nonlinear} autoregressive processes.

\subsection{Related work}

In \cite{jakubowski_2020}, it was hypothesized that the local topology of a word embedding could reflect semantic properties of the words, an idea that appears to be consistent with the data \cite{rathore2023topobert}.
Words with small local dimension (or those located near singularities) in the token subspace are expected to play linguistically significant roles.
A few papers \cite{Gromov_2024,tulchinskii2023intrinsicdimensionestimationrobust,jakubowski_2020} have derived local dimension from word embeddings.
Additionally, \cite{bradley2025magnitudecategoriestextsenriched,bradley2022enriched} shows that there is a generalized metric space in which distances between LLM outputs are determined by their probability distributions.
These papers do not address differences between possible embedding strategies, nor the fact that different LLMs will use different embeddings.
In short, while they acknowledge that topological properties can be estimated once the token subspace is found, they do not address how to find the subspace in the first place.

The transformer within an LLM makes it into a kind of dynamical system \cite{geshkovski_2023}.
A classic paper by Takens \cite{takens_1981} shows how to recover \emph{attractors}---behaviorally important subspaces---from dynamical systems.
Following the method discovered by Takens, a large literature grew around probing the internal structure of a dynamical system by building embeddings from candidate outputs, \cite{takens2006detecting,sauer1991embedology}.
While the usual understanding is that these results work for continuous time dynamical systems,
in which certain parametric choices are important \cite{xu2019twisty}, the underlying mathematical concept is \emph{transverality} \cite{Golubitsky_1973}.
Transversality is a generalization of the geometric notion of ``being in general position,'' for instance that two points in general position determine a line, or three points determine a plane.
Moreover, two submanifolds can intersect transversally, which means that along the intersection, their tangent spaces span the ambient (latent) space.
Transversality is useful because it yields strong conditions under which knowledge of several subspaces is sufficient to understand the entire ambient space.
In this paper, we apply transversality to obtain a new embedding result for general nonlinear autoregressive systems, of which LLMs are a special case.

\section{Preliminaries}
\label{sec:preliminaries}

An LLM is a kind of generalized autoregressive process, in which each timestep is a point in a fixed \emph{latent space} $X$.
Each timestep is built from the iterates of shifts $(\sigma f)$ of a given map $f : X^n \to X$ that predicts one point from a window of $n$ previous points.
In an LLM, the map $f$ is usually implemented using one or more transformer blocks, and the tokens are embedded as points in $X$.
Here, we only require that the transformer block be a smooth function, which is consistent with how they are implemented \cite{geshkovski_2023}.

\begin{definition}
  Suppose that $X$ is a smooth finite dimensional manifold of constant dimension, and that for some integer $n$, we have a smooth function $f: X^n \to X$.
  The \emph{shift} of $f$ is the function $(\sigma f) : X^n \to X^n$ given by
  \begin{equation*}
    (\sigma f)(x_1, x_2, \dotsc, x_n) := (x_2, \dotsc, x_n, f(x_1, x_2, \dotsc, x_n)).
  \end{equation*}
\end{definition}

Usually, the latent space is not made visible to the user of an LLM.
Instead, one can only obtain summary information about a point in $X$.
This could be as simple as the textual representation of a point in $X$ as a token (which is a categorical variable), but could be more detailed.
For instance, running the same query several times will yield an estimate of the probability of each token being produced.
In order to model the general setting, let us define a space $Y$ that represents the data that we can collect (say, a probability) about a token in $X$.
This measurement process is represented by a smooth function $g : X \to Y$, which---according to our Theorem \ref{thm:autoregressive_embedding}---can be chosen nearly arbitrarily.
In the context of LLMs, the function $g$ is often called the \emph{output embedding}\footnote{Caution: $g$ is rarely an ``embedding'' in the sense normally used by differential topologists.} function.

\begin{definition}
  \label{def:a_m_fg}
  The $k$-th iterate of a function $h: X\to X$, namely $k$ compositions of $h$ with itself, is written $h^{\circ k}$.  By convention $h^{\circ 0} := \id_X$, the identity function on $X$.

  Given a function $f: X^n \to X$, a function $g: X \to Y$, and a nonnegative integer $m$, the \emph{$m$-th autoregression of $f$ with measurements by $g$} is the function $\mathcal{A}_m(f,g): X^n \to Y^m$ given by
  \begin{equation*}
    \mathcal{A}_m(f,g) := \left(g \circ f,g \circ f \circ (\sigma f), g \circ f \circ (\sigma f) \circ (\sigma f), \dotsc, g \circ f \circ (\sigma f)^{\circ m-1}\right).
  \end{equation*}
  We will assume that $X$ and $Y$ are finite dimensional smooth manifolds and $f$ and $g$ are both smooth maps throughout the article.
\end{definition}

The function $\mathcal{A}_m$ represents the process of collecting data about the first $m$ tokens in the response of the LLM to a context window of length $n$.
The function $g$ represents the information we collect about a given token in the response.
Beware that in practice, since both $f$ and $g$ estimate probabilities from discrete samples, both are subject to sampling error.

While the token subspace $T$ is not generally a manifold \cite{robinson2024structure}, in practice it is always contained within a larger compact manifold,
which we will call the \emph{bounding manifold} $Z$.
Since the token subspace $T$ is not a manifold, $Z$ will generally \emph{not be equal} to the token subspace $T$.
If we obtain an embedding of $Z$, then the token subspace $T$ will also be embedded within the image of $Z$.

Our method (Algorithm \ref{alg:token_prompting}) requires that the context window of the LLM be ``cleared'' before each query to ensure that the hypotheses of Theorem \ref{thm:autoregressive_embedding} are met.
We formalize the operation of clearing the context window by considering the restriction $\mathcal{A}_m(f,g)$ to the subspace $\{x_1\} \times \dotsb \times \{x_{n-1}\} \times Z$.  Notice that this means that the first $n-1$ tokens of the initial context window are always the same (with no further constraints on exactly what values they take), while the last token in the context window is drawn from $Z$.
We write this restriction as $\mathcal{A}_m(f,g)|_{\{x_1\} \times \dotsb \times \{x_{n-1}\} \times Z} : Z \to Y^m$.

\section{Methods}
\label{sec:methods}

The main thrust of our approach is embodied by Algorithm \ref{alg:token_prompting}, which produces a set of Euclidean coordinates for each token.
Because Theorem \ref{thm:autoregressive_embedding} only yields a homeomorphism, not an isometry,
the coordinates we estimate will not be the same as those in the original embedding, nor will the distances between tokens be preserved.
Topological features, such as dimension, the presence of clusters, and (persistent) homology will nevertheless be preserved. 
Since checking whether two spaces are homeomorphic is extremely difficult \cite{ZIELINSKI2016635,Stillman_1989},
in Section \ref{sec:results}, we will only verify that dimension is preserved.
Other verifications remain as future work.

\begin{algorithm}
  \caption{Recovered token input embedding coordinates}
  \label{alg:token_prompting} 
  \begin{algorithmic}[1]
    \Require $\ell$ : number of probabilities to collect per token
    \Require $m$ : number of response tokens to collect
    \Require $x_1, \dotsc, x_{n-1}$ : fixed prefix for all queries
        
    \Procedure{TokenPromptEmbedding}{$x$; $\ell$, $m$} \Comment{$x$ is the token to embed}
    \State Clear the context window
    \State Build a query of the form $q:= (x_1, \dotsc, x_{n-1}, x)$
    \State Use the LLM to produce a response of length $m$ to the query $q$ \Comment{This computes $\mathcal{A}_m(f,\id)$}
    \State Repeat the previous three steps to estimate probability of most probable $\ell$ tokens in each of the $m$ positions \Comment{Applying the function $g$.}
    \State \Return The length $\ell m$ vector of probabilities as embedding coordinates \Comment{This is $\mathcal{A}_m(f,g)$ flattened into a single vector.}
    \EndProcedure
  \end{algorithmic}
\end{algorithm}

Figure \ref{fig:flowchart} summarizes the process.
Each token in the token set is taken as a query, and yields in response a sequence of tokens, each of which has an internal representation in $X$.
All we have access to are summary measurements of this internal representation, viewed through the function $g$.
In Algorithm \ref{alg:token_prompting}, we have chosen to define $g$ so that it estimates probabilities of tokens,
even though Theorem \ref{thm:autoregressive_embedding} is more general than that.
To that end, Steps 1-3 are repeated sufficiently many times to obtain a stable estimate of the probability of each of the $\ell$ tokens.
Therefore, each token yields a sequence of token measurements in $Y^m$, where $m$ is the number of response tokens we wish to collect.

\begin{figure}[!htbp]
  \begin{center}
    \includegraphics[width=5in]{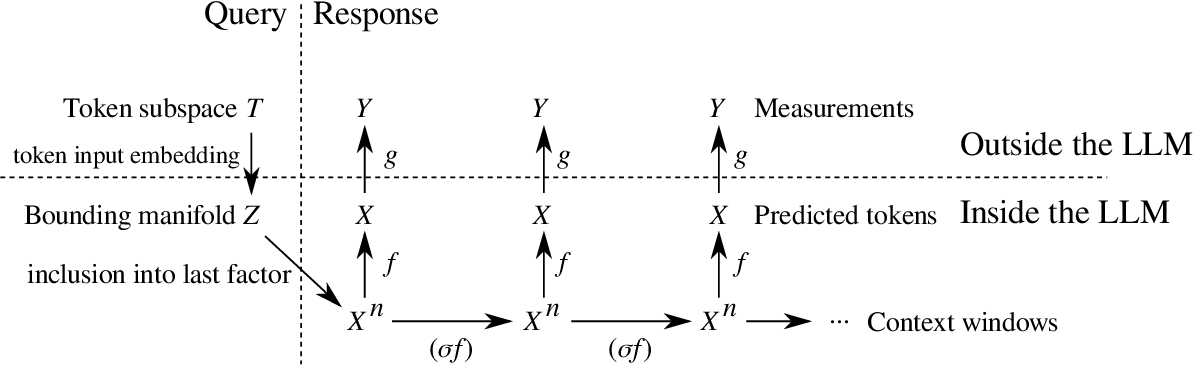}
    \caption{Flowchart of our Algorithm \ref{alg:token_prompting}: queries consist of individual tokens entering from the left of the frame, and result in a stream of measurements of tokens from the right.  Briefly, $f$ represents the action of the transformer blocks of the LLM, $(\sigma f)$ updates the context window between tokens, and $g$ is the output embedding, which produces probabilities for each of the tokens.}
    \label{fig:flowchart}
  \end{center}
\end{figure}

The correctness of Algorithm \ref{alg:token_prompting} is justified by Theorem \ref{thm:autoregressive_embedding},
which asserts that the process of collecting summary information about the sequence of tokens generated in response to single-token queries is an embedding provided certain bounds on the number of tokens collected are met.

\begin{theorem} (proven in the Appendix)
  \label{thm:autoregressive_embedding}
  Suppose that $X$ is a smooth manifold, $Y$ is a smooth manifold of dimension $\ell$, that $x_1, \dotsc x_{n-1}$ are elements of $X$, and $Z$ is a submanifold of $X$ of dimension $d$.
  For smooth functions $f : X^n \to X$ and $g: X \to Y$, the function $\mathcal{A}_m(f,g) : X^n \to Y^m$ given in Definition \ref{def:a_m_fg}, collects $m$ samples from iterates of $(\sigma f)$.

  If the dimensions of the above manifolds are chosen such that
  \begin{equation*}
    2d < m \min\{\ell,\dim X\} \le n \min\{\ell,\dim X\},
  \end{equation*}
  then there is a residual subset\footnote{A \emph{residual subset} is the intersection of countably many open and dense subsets.} $V$ of $C^\infty(X,Y)$
  such that if  $g \in V$, then there is a (different) residual subset $U$ of $C^\infty(X^n,X)$ such that if $f \in U$, then the function
  \begin{equation*}
    \mathcal{A}_m(f,g)|_{\{x_1\} \times \dotsb \times \{x_{n-1}\} \times Z} : Z \to Y^m
  \end{equation*}
  is a smooth injective immersion\footnote{An \emph{immersion} is a smooth function whose derivative (Jacobian matrix) is injective at all points.}.

  Following this, it is a standard fact (see \cite[Prop. 7.4]{Lee_2003}) that if $Z$ is compact, then $\mathcal{A}_m(f,g)$ is an embedding of $Z$ into $Y^m$.
\end{theorem}

Since the token subspace $T$ is always compact, if we satisfy the inequalities on dimensions, most choices of $f$ (the LLM) and $g$ (the measurements we collect) will yield embeddings of $T \subseteq Z$ into $Y^m$ if the appropriate number of measurements is collected.
Furthermore, since embeddings induce homeomorphisms on their images, even if the token subspace is not a manifold, its image within $Y^m$ will be topologically unchanged.

The ordering of the residual subsets, that first $g$ is chosen and then $f$ may be chosen according to a constraint imposed by $g$, is necessary to make the proof of Theorem \ref{thm:autoregressive_embedding} work.  In essence, a poor choice of $g$ can preclude making \emph{any} useful measurements, regardless of $f$.
Although this is opposite to the way Algorithm \ref{alg:token_prompting} operates---the LLM (described by $f$) is selected without regard for the properties to be collected (described by $g$)---this difference is likely not terribly important given that a poor choice of $g$ is not likely to recommend itself to engineers in the first place.

\section{Results}
\label{sec:results}

To demonstrate our method, we chose to work with {\tt Llemma-7B} \cite{llemma}.
This model has $32016$ tokens in total, which are embedded in a latent space of dimension $\dim X = 4096$.
Since the source code and pre-trained weights for the model are available, the token input embedding is known.
We can therefore compare the token subspace of the model with the embedding computed by Algorithm \ref{alg:token_prompting}.
{\tt Llemma-7B} uses a context window of $n=4096$ tokens, since \cite{azerbayev2023llemma} says the model was trained on sequences of this length.

As per the results exhibited in \cite[Fig 7]{robinson2024structure}, the token subspace is a stratified manifold in which all strata are dimension $14$ or less.
Therefore, the manifold $Z$ that contains the token subspace can be taken to be not more than $d=28$ dimensional, and even though this is probably quite a loose bound.

Algorithm \ref{alg:token_prompting} requires the use of a function $g$ to collect $m$ samples from each response token.
We tried three different Options for the choice of $g$ and the number of tokens we collected:
\begin{description}
\item[Option (1):] Collect $m=30$ response tokens and $\ell = 3$ probabilities for the top three tokens at each response token position (ignoring what the tokens actually were),
\item[Option (2):] Collect $m=30$ response tokens and $\ell = 32016$ probabilities, one for each token, but aggregated over the entire response, and
\item[Option (3):] Collect $m=1$ response token and $\ell = 32016$ probabilities, one for each token being the first token in the response.
\end{description}

Option (1) satisfies the hypotheses of Theorem \ref{thm:autoregressive_embedding} because
\begin{equation*}
  2d = 56 \le m \ell = 30 \times 3 = 90  \le n \ell = 4096 \times 3 = 12288.
\end{equation*}

Option (2) also satisfies the hypotheses of Theorem \ref{thm:autoregressive_embedding} because
\begin{equation*}
  2d = 56 \le m \ell = 30 \times \min\{ 4096, 32016\}  \le n \ell = 4096 \times \min\{ 4096, 32016\}.
\end{equation*}

Option (3) also satisfies the hypotheses of Theorem \ref{thm:autoregressive_embedding} because
\begin{equation*}
  2d = 56 \le m \ell = 1 \times \min\{ 4096, 32016\}  \le n \ell = 4096 \times \min\{ 4096, 32016\}.
\end{equation*}

Given these experimental parameters, we used Algorithm \ref{alg:token_prompting} to compute coordinates for each of the 32016 tokens.  This process took 13 hours of walltime on a Macbook M3.
Afterwards, the method from \cite{robinson2024structure} was used to estimate the local dimension at each token; this took 3 hours of walltime on a Core i7-3820 running 3.60GHz.
The entire process, consisting of the pipeline of (1) selection of one of the Options, (2) Algorithm \ref{alg:token_prompting} using that Option for data collection, followed by (3) dimension estimation, will be referred to a ``proposed dimension estimator''.

Because the embedding produced by Algorithm \ref{alg:token_prompting} using Option (1) is into $\mathbb{R}^{90}$ while the others embed into substantially higher dimensional space, doing any postprocessing is much easier on Option (1) than the others.
As a consequence, we will mostly focus on the proposed dimension estimator with Option (1) after establishing that it is qualitatively similar to the others.

\begin{figure}[!htbp]
  \begin{center}
    \includegraphics[width=5in]{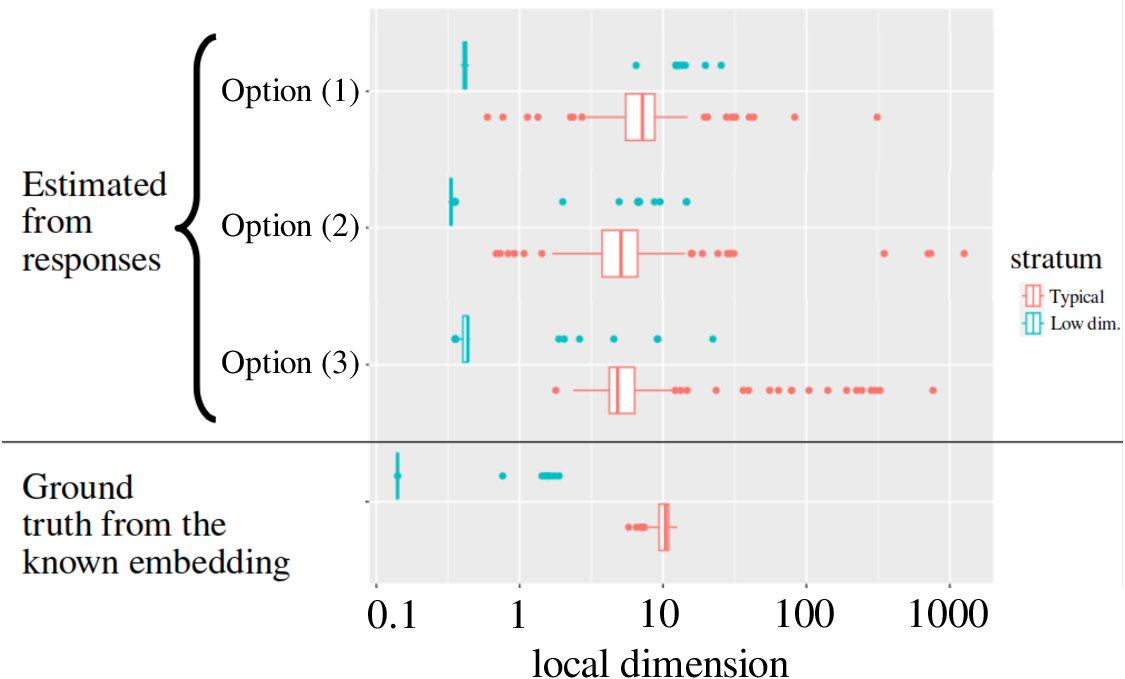}
    \caption{Comparison of estimated dimension on a stratified sample using each of the proposed dimension estimators and the dimension estimated directly from the embedding.  Note: the local dimension from the known embedding is the ``base'' dimension, not the fiber dimension; see Figure \ref{fig:sample_v_r}.}
    \label{fig:boxplots}
  \end{center}
\end{figure}

Because Options (2) and (3) are very computationally expensive, we drew a stratified sample of 200 random tokens to compare dimensions across all three Options and the original embedding.
We drew a simple random sample of 100 tokens from the distribution at large and another simple random sample of 100 tokens with known low dimension (below 5).
The reason for this particular choice is that \cite{robinson2024structure} establishes that a few tokens in {\tt Llemma-7B} have unusually low dimension (less than 5), and we would like to ensure that this is correctly captured.

Figure \ref{fig:boxplots} shows the estimated dimension for these two strata across all three Options against the true token input embedding.
It is clear that both strata are recovered by Algorithm \ref{alg:token_prompting} using all three Options,
but that there are biases present and the variance of the estimates is increased.
This is largely due to the fact that probability estimation is subject to sampling error, both in the $g$ function and the LLM itself.
Figure \ref{fig:boxplots} indicates that the differences between different Options are not large.
In short, the token input embedding has a strong impact upon the behavior of the LLM.

\begin{figure}[!htbp]
  \begin{center}
    \includegraphics[width=5in]{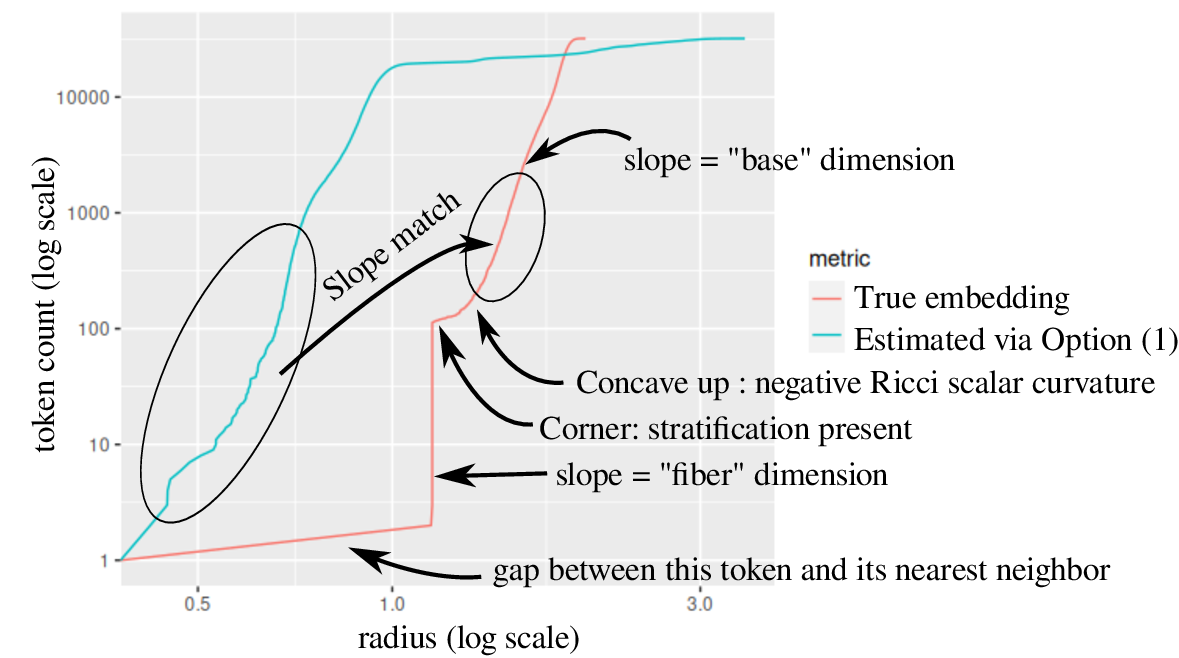}
    \caption{The log-log volume versus radius plot for the token ``{\tt \}}'' at start of a word obtained from the original embedding (red) and the proposed dimension estimator with Option (1) (blue).} 
    \label{fig:sample_v_r}
  \end{center}
\end{figure}

Following the methodology of \cite{robinson2024structure}, we recognize that dimension estimates do not tell the full story.
We examined the volume (token count) versus radius for all of the tokens.  Figure \ref{fig:sample_v_r} shows one instance, which corresponds to the token ``{\tt \}}'' appearing at the start of a word.  Notice that there are ``corners'' present in the curve derived from the original embedding, which is indicative of stratifications in the token subspace.  In fact, most of the tokens in the token subspace for {\tt Llemma-7B} have \emph{two} salient dimensions: one for small radii and one for larger radii.  The stratification structure in the vicinity of this particular token is indicative of a negatively curved stratum (larger radii) that has been thickened by taking the cartesian product with a high dimensional sphere of definite radius (smaller radii).  In the case of Figure \ref{fig:sample_v_r}, the radius of this sphere is approximately $1.2$, as indicated by the vertical portion of the red curve.

The structure of two distinct dimensions is typical throughout the token subspace for {\tt Llemma-7B}, though the radius of the inner sphere tends to vary substantially.
As such, the token subspace seems to have the structure of a \emph{fiber bundle} over a stratified manifold at almost all tokens; it is locally homeomorphic to a cartesian product of a sphere and a stratified manifold except at a small number of tokens.  We will call the lower dimensional stratum---which occurs for larger radii---the \emph{base} stratum, and the higher dimensional spherical stratum---which occurs for smaller radii---the \emph{fiber} stratum in what follows.
One should use the intuition that the base stratum corresponds to the inherent semantic variability in the tokens, while the fiber stratum is mostly capturing model uncertainty, noise, and other effects.  Indeed, the definite radius of the sphere at each token appears to be characteristic of {\tt Llemma-7B}; other LLMs do not seem to exhibit this structure as clearly.

Notice that for the token shown in Figure \ref{fig:sample_v_r}, there is an apparent match in slope between the estimate provided by the proposed dimension estimator with Option (1) and the slope within the base portion of the original embedding.  This suggests that the base portion of the token subspace is most important in terms of the LLM's responses.  This accords with some of the intrinsic dimension estimates in the literature for natural language (see for instance \cite{Gromov_2024,tulchinskii2023intrinsicdimensionestimationrobust,jakubowski_2020}), since the base dimension tends to be in the vicinity of 5--10, whereas the fiber is much higher dimensional.  This intuition is confirmed in the analysis that follows.  The dimension estimated by the proposed dimension estimator across all three Options is more sensitive to the base dimension, and therefore can be understood to estimate the base space of the fiber bundle structure of the token subspace.

\begin{figure}[!htbp]
  \begin{center}
    \includegraphics[width=4in]{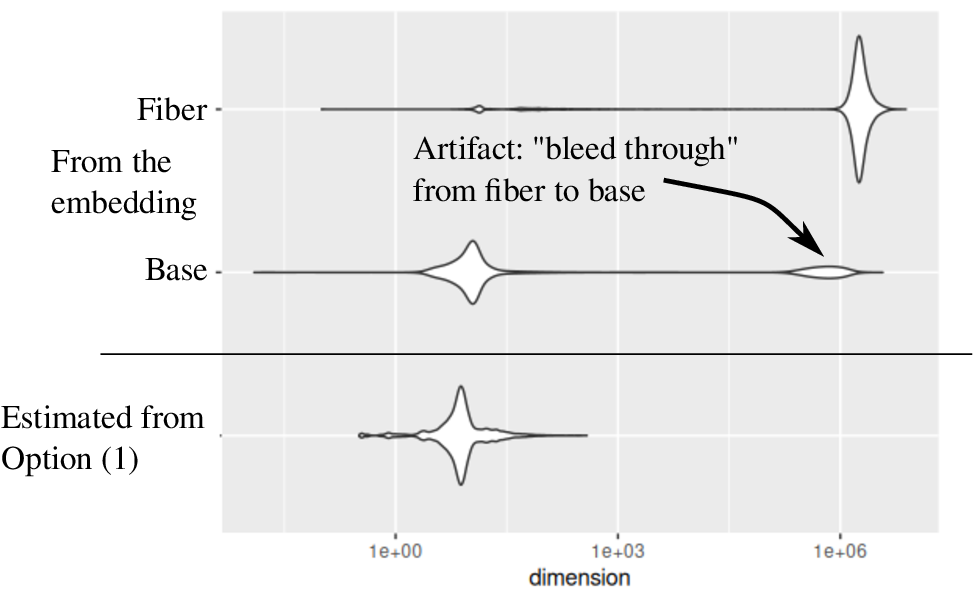}
    \caption{Histograms of the local dimensions estimated for all tokens.}
    \label{fig:b_f_e}
  \end{center}
\end{figure}

Figure \ref{fig:b_f_e} shows the distribution of local dimensions estimated for all tokens using the original embedding (base and fiber computed separately) and using the proposed dimension estimator with Option (1).
As is clear from Figure \ref{fig:sample_v_r}, stratifications are not clearly visible with the proposed dimension estimator, so each token is ascribed only one dimension.
Figure \ref{fig:b_f_e} shows that the estimated dimension aligns closely with the base dimension.
This agrees with the intuition that the fiber dimension mostly corresponds to noiselike components, which ultimately have less impact on the LLM responses.

\begin{figure}[!htbp]
  \begin{center}
    \includegraphics[width=5in]{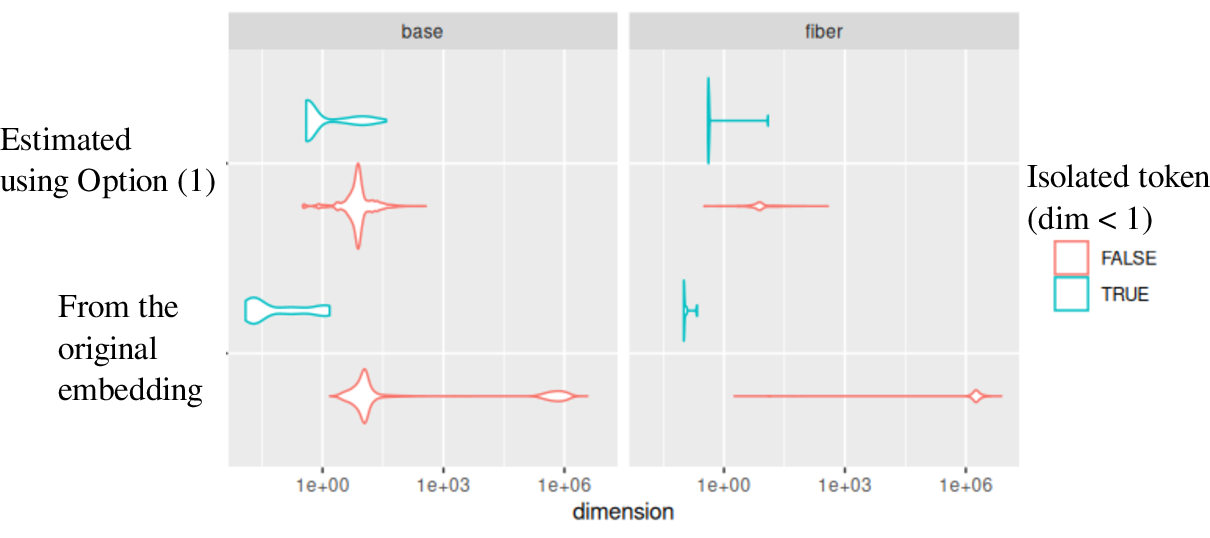}
    \caption{Comparison of estimated dimension using the proposed dimension estimator with Option (1) and the dimension estimated directly from the embedding}
    \label{fig:option_1_box}
  \end{center}
\end{figure}

There are a few tokens with very few neighbors under the token input embedding.  These \emph{isolated tokens} have local dimension less than $1$.  Figure \ref{fig:option_1_box} shows that the proposed dimension estimator with Option (1) reports significantly lower dimensions for the isolated tokens, both in the fiber and the base.
There is a bias towards the median dimension of approximately $10$ that is likely due to random sampling effects.

\section{Discussion}
\label{sec:discussion}

We have demonstrated that Algorithm \ref{alg:token_prompting} allows one to impute coordinates to tokens from LLM responses in a way that leads to an embedding of the latent space into the space of responses.
This establishes that the topology of an LLM's token subspace has a strong link to the LLM's behavior,
specifically tokens that are near each other result in similar responses.
As suggested by others \cite{robinson2024structure,jakubowski_2020} this may explain why LLMs are not performant against certain kinds of queries.

\begin{table}
  \caption{Sample responses to queries}
  \label{tab:sample_responses}
  \begin{tabular}{|c||l|}
    \hline
    Token & Response \\
    \hline
    \hline
        {\tt <unk>} & {\tt 1.1k views\textbackslash nConsider the following statements: \textbackslash n1. } \\ 
        & {\tt The complement of every Turing decidable language is} \\ 
        & {\tt Turing decid} \\ 
        \hline
            {\tt <0xC5>} & {\tt  15:10, 11 September 2008 (UTC)\textbackslash n\#\#\# 2008}\\ 
            \hline
                {\tt <0xC6>} & {\tt  10000000000000000000000000000}\\ 
                \hline
                    {\tt ot} & {\tt ally, the same as the one in the previous section.\textbackslash n}\\ 
                & {\tt \textbackslash begin\{figure\}[htbp]\textbackslash n\textbackslash centering\textbackslash n\textbackslash includegraphics[width}\\ 
                       
                      \hline
    {\tt \}}& (string consisting entirely of spaces)\\ 
    \hline
  \end{tabular}
\end{table}

Aside from the process of estimating the topology of the token subspace,
the raw output of Algorithm \ref{alg:token_prompting} is interesting on its own.
Table \ref{tab:sample_responses} shows just a few of the most intriguing results.
Each of the responses shown were produced from single token queries, yet 
some of the output of {\tt Llemma-7B} in response to a single token is apparently coherent.
The responses to semantically meaningless queries can even include (as in the case of the last entry of Table \ref{tab:sample_responses}) syntactically correct source code.
This suggests that the data collection part of Algorithm \ref{alg:token_prompting} may be finding interesting \emph{glitch tokens} \cite{li2024glitch,zhang2024glitchprober,wu2024mining,land2024fishing}.

\section*{Acknowledgements}

The authors would like to thank Andrew Lauziere for his helpful suggestions on improving a draft of this paper.

This article is based upon work supported by the Defense Advanced Research Projects Agency (DARPA).
Any opinions, findings and conclusions, or recommendations expressed in this material are those of the authors and do not necessarily reflect the views of DARPA.

\bibliographystyle{unsrt}
\bibliography{takenstokens_bib}

\begin{thebibliography}{10}

\bibitem{robinson2024structure}
Michael Robinson, Sourya Dey, and Shauna Sweet.
\newblock The structure of the token space for large language models.
\newblock {\em arXiv preprint arXiv:2410.08993}, 2024.

\bibitem{jakubowski_2020}
Alexander Jakubowski, Milica Gasic, and Marcus Zibrowius.
\newblock Topology of word embeddings: Singularities reflect polysemy.
\newblock In Iryna Gurevych, Marianna Apidianaki, and Manaal Faruqui, editors,
  {\em Proceedings of the Ninth Joint Conference on Lexical and Computational
  Semantics}, pages 103--113, Barcelona, Spain (Online), December 2020.
  Association for Computational Linguistics.

\bibitem{rathore2023topobert}
Archit Rathore, Yichu Zhou, Vivek Srikumar, and Bei Wang.
\newblock Topobert: Exploring the topology of fine-tuned word representations.
\newblock {\em Information Visualization}, 22(3):186--208, 2023.

\bibitem{Gromov_2024}
Vasilii~A. Gromov, Nikita~S. Borodin, and Asel~S. Yerbolova.
\newblock A language and its dimensions: Intrinsic dimensions of language
  fractal structures.
\newblock {\em Complexity}, 2024.

\bibitem{tulchinskii2023intrinsicdimensionestimationrobust}
Eduard Tulchinskii, Kristian Kuznetsov, Laida Kushnareva, Daniil Cherniavskii,
  Serguei Barannikov, Irina Piontkovskaya, Sergey Nikolenko, and Evgeny
  Burnaev.
\newblock Intrinsic dimension estimation for robust detection of ai-generated
  texts, 2023.

\bibitem{bradley2025magnitudecategoriestextsenriched}
Tai-Danae Bradley and Juan~Pablo Vigneaux.
\newblock The magnitude of categories of texts enriched by language models,
  2025.

\bibitem{bradley2022enriched}
Tai-Danae Bradley, John Terilla, and Yiannis Vlassopoulos.
\newblock An enriched category theory of language: from syntax to semantics.
\newblock {\em La Matematica}, 1(2):551--580, 2022.

\bibitem{geshkovski_2023}
Borjan Geshkovski, Cyril Letrouit, Yury Polyanskiy, and Philippe Rigollet.
\newblock A mathematical perspective on transformers, 2023.

\bibitem{takens_1981}
Floris Takens.
\newblock Detecting strange attractors in turbulence.
\newblock {\em Springer Lecture Notes in Mathematics}, 898:366--381, 1981.

\bibitem{takens2006detecting}
Floris Takens.
\newblock Detecting strange attractors in turbulence.
\newblock In {\em Dynamical Systems and Turbulence, Warwick 1980: proceedings
  of a symposium held at the University of Warwick 1979/80}, pages 366--381.
  Springer, 2006.

\bibitem{sauer1991embedology}
Tim Sauer, James~A Yorke, and Martin Casdagli.
\newblock Embedology.
\newblock {\em Journal of statistical Physics}, 65:579--616, 1991.

\bibitem{xu2019twisty}
Boyan Xu, Christopher~J Tralie, Alice Antia, Michael Lin, and Jose~A Perea.
\newblock Twisty takens: A geometric characterization of good observations on
  dense trajectories.
\newblock {\em Journal of Applied and Computational Topology}, 3:285--313,
  2019.

\bibitem{Golubitsky_1973}
M.~Golubitsky and V.~Guillemin.
\newblock {\em Stable Mappings and Their Singularities}.
\newblock Springer, 1973.

\bibitem{ZIELINSKI2016635}
Joseph Zielinski.
\newblock The complexity of the homeomorphism relation between compact metric
  spaces.
\newblock {\em Advances in Mathematics}, 291:635--645, 2016.

\bibitem{Stillman_1989}
Jonathan Stillman.
\newblock {\em Computational problems in equational theorem proving}.
\newblock PhD thesis, State University of New York at Albany, USA, 1989.
\newblock AAI9016549.

\bibitem{Lee_2003}
John~M. Lee.
\newblock {\em Introduction to smooth manifolds}.
\newblock Springer, 2003.

\bibitem{llemma}
Llemma-7b \url{https://huggingface.co/EleutherAI/llemma_7b}, March 2024.

\bibitem{azerbayev2023llemma}
Zhangir Azerbayev, Hailey Schoelkopf, Keiran Paster, Marco~Dos Santos, Stephen
  McAleer, Albert~Q Jiang, Jia Deng, Stella Biderman, and Sean Welleck.
\newblock Llemma: An open language model for mathematics.
\newblock {\em arXiv preprint arXiv:2310.10631}, 2023.

\bibitem{li2024glitch}
Yuxi Li, Yi~Liu, Gelei Deng, Ying Zhang, Wenjia Song, Ling Shi, Kailong Wang,
  Yuekang Li, Yang Liu, and Haoyu Wang.
\newblock Glitch tokens in large language models: categorization taxonomy and
  effective detection.
\newblock {\em Proceedings of the ACM on Software Engineering},
  1(FSE):2075--2097, 2024.

\bibitem{zhang2024glitchprober}
Zhibo Zhang, Wuxia Bai, Yuxi Li, Mark~Huasong Meng, Kailong Wang, Ling Shi,
  Li~Li, Jun Wang, and Haoyu Wang.
\newblock Glitchprober: Advancing effective detection and mitigation of glitch
  tokens in large language models.
\newblock In {\em Proceedings of the 39th IEEE/ACM International Conference on
  Automated Software Engineering}, pages 643--655, 2024.

\bibitem{wu2024mining}
Zihui Wu, Haichang Gao, Ping Wang, Shudong Zhang, Zhaoxiang Liu, and Shiguo
  Lian.
\newblock Mining glitch tokens in large language models via gradient-based
  discrete optimization.
\newblock {\em arXiv preprint arXiv:2410.15052}, 2024.

\bibitem{land2024fishing}
Sander Land and Max Bartolo.
\newblock Fishing for magikarp: Automatically detecting under-trained tokens in
  large language models.
\newblock {\em arXiv preprint arXiv:2405.05417}, 2024.

\end{thebibliography}

\section*{Appendix: embedding methods for autoregressive processes}
\label{sec:appendix}

Theorem \ref{thm:autoregressive_embedding} bounds the number of iterates to be gathered in order to recover the token subspace up to homeomorphism.
Interestingly, depending on the size of the context window $n$ and the dimension bound $d$ on the token subspace, it may be impossible to recover the token subspace.
However, since the trend is to use large context windows, recovery is usually possible.

Toward the goal of proving Theorem \ref{thm:autoregressive_embedding}, a preliminary fact is that bijectivity of the shift of $f$ depends only on the first coordinate of $f$.
\begin{proposition}
  \label{prop:bijective_shifts}
  The shift of $f$ is bijective if and only if for every $(x_2, \dotsc, x_n) \in X^{n-1}$, $f|_{X \times \{x_2\} \times \dotsb \times \{x_n\}}: X \to X$ is bijective.
\end{proposition}
\begin{proof}
  Suppose that $(\sigma f)(x_1, x_2, \dotsc, x_n) = (\sigma f)(y_1, y_2, \dotsc, y_n)$.  This necessarily establishes the equality of $x_2 = y_2$, $\dotsc$, $x_n = y_n$.
  This means that $f(x_1, x_2, \dotsc, x_n) = f(y_1, x_2, \dotsc, x_n)$.
  Clearly if $f$ is injective on its first coordinate, then $x_1 = y_1$.
  Conversely if $(\sigma f)$ is injective, then we must have that $f$ is injective on its first coordinate.

  Surjectivity is similar: requiring a particular solution to $(\sigma f)(x_1, \dotsc, x_n) = (z_1, \dotsc, z_n)$ means that we must have $x_2 = z_1$, $\dotsc$, $x_n = z_{n-1}$.
  Therefore $z_n = f(x_1, z_1, \dotsc, z_{n-1})$, which immediately requires surjectivity of $f$ when restricted to its first coordinate.
  Conversely, if $f$ is surjective when restricted to its first coordinate, surjectivity of $(\sigma f)$ follows.
\end{proof}

Although the shift of $f$ is well defined if $X$ is a smooth manifold, an important special case is when $X = \mathbb{R}^d$.
This will play a role in what follows because we will work locally, and the tangent spaces of a manifold are vector spaces.

\begin{lemma}
  \label{lem:sigma_f_block}
  If $X = \mathbb{R}^d$ and $f$ is a linear map, then $(\sigma f)$ has the block matrix representation
  \begin{equation*}
    (\sigma f) = S_n + F = \begin{pmatrix}
      0_{d \times d}      & \id_X  & 0_{d \times d}      & \dotsb & 0_{d \times d} \\
             & \ddots & \ddots & \ddots & \vdots \\
      \vdots &        & \ddots & \ddots & 0_{d \times d} \\
             &        &        & \ddots & \id_X \\
      0_{d \times d}      &        & \dotsb &        & 0_{d \times d} \\
    \end{pmatrix} +
    \begin{pmatrix} 0_{(n-1)d \times nd} & \\ \vdots \\ f \end{pmatrix}.
  \end{equation*}
\end{lemma}

By repeated application of the chain rule, we can compute the derivative of $\mathcal{A}_m(f,g)$,
\begin{equation*}
  \begin{aligned}
    \left(d \mathcal{A}_m(f,g)\right)(x) & = \begin{pmatrix} d_x(g\circ f) \\
      d_x\left(g \circ f \circ (\sigma f)\right) \\
      d_x\left(g \circ f \circ (\sigma f) \circ (\sigma f)\right) \\
      \vdots \\
      d_x\left(g \circ f \circ (\sigma f)^{\circ m-1} \right)\end{pmatrix}\\
    & = \begin{pmatrix} d_{f(x)}g \circ d_xf \\
      d_{(f \circ (\sigma f))(x)}g \circ d_{(\sigma f)(x)}f \circ d_x(\sigma f) \\
      d_{(f \circ (\sigma f) \circ (\sigma f))(x)}g \circ d_{((\sigma f) \circ (\sigma f))(x)}f \circ d_{(\sigma f)(x)}(\sigma f) \circ d_x(\sigma f) \\
      \vdots \\
      d_{(f \circ (\sigma f)^{\circ m-1})(x)}g \circ d_{(\sigma f)^{\circ m-1}(x)}f \circ d_{(\sigma f)^{\circ m-2}(x)}(\sigma f) \dotsb d_x(\sigma f)\end{pmatrix}.\\
  \end{aligned}
\end{equation*}

\begin{lemma}
  \label{lem:d_sigma_f}
  Consider the restriction of a smooth function $f : X^n \to X$ to the first component of $X^n$ at a given point $x = (x_1, x_2, \dotsc, x_n)$ of $X^n$.
  This is the function
  \begin{equation*}
    \left(f_{X \times \{x_2\} \times \dotsb \times \{x_n\}}\right) : X \to X.
  \end{equation*}
  Under this situation,
  \begin{equation*}
    \rank d_x(\sigma f)=(\dim X)(n-1) + \rank d_{x_1} \left(f_{X \times \{x_2\} \times \dotsb \times \{x_n\}}\right).
  \end{equation*}
\end{lemma}
\begin{proof}
  As a illustrative shorthand, write
  \begin{equation*}
    \frac{\partial f}{\partial x_k}(x) = d_{x_k} \left(f_{\{x_1\} \times \dotsb X \times \dotsb \times \{x_n\}}\right),
  \end{equation*}
  where the right side is the derivative (Jacobian matrix) of $f$ restricted to the $k$-th component of $X^n$.

  We can proceed by computing $d_x(\sigma f)$ in block matrix form if we assume that $\dim X = d$, which accords with Lemma \ref{lem:sigma_f_block}:
  \begin{equation*}
    \begin{aligned}
      d_x(\sigma f) & = d_{(x_1, x_2, \dotsc, x_n)} (x_2, \dotsc, x_n, f(x_1, x_2, \dotsc, x_n)) \\
      & = \begin{pmatrix}d_{(x_1, x_2, \dotsc, x_n)} x_2 \\  \vdots \\ d_{(x_1, x_2, \dotsc, x_n)} x_n \\ d_{(x_1, x_2, \dotsc, x_n)}f(x_1, x_2, \dotsc, x_n)) \end{pmatrix}\\
      & = \begin{pmatrix} 
        0_{d \times d}      & \id_X  & 0_{d \times d}      & \dotsb & 0_{d \times d} \\
        & \ddots & \ddots & \ddots & \vdots \\
        \vdots &        & \ddots & \ddots & 0_{d \times d} \\
         0_{d \times d }&  \dotsb      &        & 0_{d \times d} & \id_X \\
        \frac{\partial f}{\partial x_1}(x)&\frac{\partial f}{\partial x_2}(x)   &  \dotsb      & & \frac{\partial f}{\partial x_n}(x) \\
        \end{pmatrix}.\\
    \end{aligned}
    \end{equation*}
    The upper portion of the above matrix is of full rank, namely $(\dim X)(n-1)$, as the rows are all linearly independent.
    Therefore, the rank of $d_x(\sigma f)$ depends entirely upon the rank of $\frac{\partial f}{\partial x_1}(x)$.
\end{proof}

\begin{lemma}
  \label{lem:autoregressive_immersion}
  Suppose $x_1, \dotsc x_{n-1}$ are elements of $X$.  If $m \ge 2$, there is a residual subset $U$ of $C^\infty(X^n,X)$, such that if $f \in U$ then $\mathcal{A}_m(f,\id_X)|_{\{x_1\} \times \dotsb \times \{x_{n-1}\} \times X} : X \to X^m$ is a smooth immersion.
\end{lemma}
\begin{proof}
  It suffices to show the statement for $m=2$.
  We compute the derivative of $\mathcal{A}_2(f,\id_X)$ completely, using the notation from Lemma \ref{lem:d_sigma_f}.
  To that end, we start with computing
  \begin{equation*}
    \begin{aligned}
      d_{(\sigma f)(x)}f \circ d_x(\sigma f) &= \begin{pmatrix}\frac{\partial f}{\partial x_1}\left((\sigma f)(x)\right)& \dotsb & \frac{\partial f}{\partial x_n}\left((\sigma f)(x)\right)\end{pmatrix}
      \left(
      \begin{smallmatrix} 
        0_{d \times d}      & \id_X  & 0_{d \times d}      & \dotsb & 0_{d \times d} \\
        & \ddots & \ddots & \ddots & \vdots \\
        \vdots &        & \ddots & \ddots & 0_{d \times d} \\
         0_{d \times d }&  \dotsb      &        & 0_{d \times d} & \id_X \\
        \frac{\partial f}{\partial x_1}(x)&\frac{\partial f}{\partial x_2}(x)   &  \dotsb      & & \frac{\partial f}{\partial x_n}(x) \\
      \end{smallmatrix} \right)\\
      &= \left(\begin{smallmatrix}
        \frac{\partial f}{\partial x_n}\left((\sigma f)(x)\right)\frac{\partial f}{\partial x_1}(x)&  \dotsb      & \frac{\partial f}{\partial x_n}\left((\sigma f)(x)\right)\frac{\partial f}{\partial x_{n-1}}(x) & \frac{\partial f}{\partial x_{n-1}}\left((\sigma f)(x)\right) + \frac{\partial f}{\partial x_n}\left((\sigma f)(x)\right)\frac{\partial f}{\partial x_n}(x)
        \end{smallmatrix}\right)
    \end{aligned}
  \end{equation*}

  Given the above calculation, we can can compute  
  \begin{equation*}
    \begin{aligned}
    d_x \left(\mathcal{A}_2(f,\id_X)\right)(x) &= \begin{pmatrix} d_xf \\
      d_{(\sigma f)(x)}f \circ d_x(\sigma f) \\
    \end{pmatrix} \\
    &= \left(\begin{smallmatrix}
      \frac{\partial f}{\partial x_1}(x) & \dotsb & \frac{\partial f}{\partial x_{n-1}}(x) & \frac{\partial f}{\partial x_n}(x)\\
      \frac{\partial f}{\partial x_n}\left((\sigma f)(x)\right)\frac{\partial f}{\partial x_1}(x)&  \dotsb      & \frac{\partial f}{\partial x_n}\left((\sigma f)(x)\right)\frac{\partial f}{\partial x_{n-1}}(x) & \frac{\partial f}{\partial x_{n-1}}\left((\sigma f)(x)\right) + \frac{\partial f}{\partial x_n}\left((\sigma f)(x)\right)\frac{\partial f}{\partial x_n}(x)\\
      \end{smallmatrix}\right)\\
    \end{aligned}
  \end{equation*}
  
  Recognize that being a smooth immersion means that the rank of the derivative of $\mathcal{A}_m(f,\id_X)|_{\{x_1\} \times \dotsb \times \{x_{n-1}\} \times X}$ at every point must be nonsingular.
  This means we really only need to consider the last column of the above, so that
  \begin{equation*}
    \begin{aligned}
    d\left(\mathcal{A}_m(f,\id_X)|_{\{x_1\} \times \dotsb \times \{x_{n-1}\} \times X}\right)(x_n) &= \begin{pmatrix}
      \frac{\partial f}{\partial x_n}(x)\\
      \frac{\partial f}{\partial x_{n-1}}\left((\sigma f)(x)\right) + \frac{\partial f}{\partial x_n}\left((\sigma f)(x)\right)\frac{\partial f}{\partial x_n}(x)\\
    \end{pmatrix}\\
    &= \begin{pmatrix} \id_x & 0_{d \times d} \\\frac{\partial f}{\partial x_n}\left((\sigma f)(x)\right)  & \id_X \end{pmatrix} \begin{pmatrix} \frac{\partial f}{\partial x_n}(x) \\ \frac{\partial f}{\partial x_{n-1}}((\sigma f)(x)) \end{pmatrix}
    \end{aligned}
  \end{equation*}
  Observe that the notation $\frac{\partial f}{\partial x_{n-1}}$ means the Jacobian matrix of partial derivatives in the $(n-1)$-th component, which is \emph{not necessarily} the original value of $x_{n-1}$.  Given this caution, 
  \begin{equation*}
    \begin{aligned}
      \frac{\partial f}{\partial x_{n-1}}\left((\sigma f)(x)\right) &= \frac{\partial }{\partial x_{n-1}}\left( f(x_2, \dotsc, x_n, f(x_1, x_2, \dotsc, x_n))  \right) \\
      &= d_{(x_2, \dotsc, x_n, f(x_1, x_2, \dotsc, x_n))} f|_{\{x_2\} \times \dotsc \{x_{n-1}\} \times X \times \{f(x)\}} \\
      &=d_{(\sigma f)(x)} f|_{\{x_2\} \times \dotsc \{x_{n-1}\} \times X \times \{f(x)\}} \\
    \end{aligned}
  \end{equation*}
  In short, the conditions imposed by the upper block of
  \begin{equation*}
    d\left(\mathcal{A}_m(f,\id_X)|_{\{x_1\} \times \dotsb \times \{x_{n-1}\} \times X}\right)(x_n)
  \end{equation*}
  and the lower block apply to two different restrictions of $f$ to subspaces.
  
  The above matrix is of size $2\dim X \times \dim X$.
  To obtain an immersion, the rank of this matrix must be $\dim X$.
  According to \cite[Lemma 5.1]{Golubitsky_1973} and the above calculation, we will have an immersion if the derivative of $f$ restricted to $\{x_1\} \times \dotsb \times \{x_{n-2}\} \times X \times X$ does not intersect the subspace $K$ of matrices with rank strictly less than $\dim X$ in each fiber.
  According to \cite[Prop 5.3]{Golubitsky_1973}, this subspace of $2 \dim X \times \dim X$ matrices of rank $\dim X - r$ for $r \ge 1$ is a submanifold of codimension
  \begin{equation*}
    (2 \dim X - \dim X + r)(\dim X - \dim X + r) = (\dim X +r) r \ge \dim X + 1.
  \end{equation*}
     However, Thom transversality \cite[Thm 4.9]{Golubitsky_1973} yields a residual subset $U$ of $C^\infty(X^2,X)$ such that if $f \in U$, then the derivative of $f$ (restricted to the last two coordinates) is transverse to $K$.
   According to \cite[Thm 4.4]{Golubitsky_1973}, this means that for $f \in U$, $f^{-1}(K)$ is of codimension at least $\dim X+1$,
   which finally, according to \cite[Prop 4.3]{Golubitsky_1973} means that when we restrict $f$ to its last coordinate, its image will not intersect $K$.
\end{proof}

If we acquire another sample, we can ascribe injectivity to our immersion.

\begin{lemma}
  \label{lem:autoregressive_embedding_3}
  Suppose $x_1, \dotsc x_{n-1}$ are elements of $X$.  There is a residual subset $V$ of $C^\infty(X^n,X)$, such that if $f \in V$ then
  \begin{equation*}
    \mathcal{A}_m(f,\id_X)|_{\{x_1\} \times \dotsb \times \{x_{n-1}\} \times X} : X \to X^m
  \end{equation*}
  is a smooth injective immersion whenever $m \ge 3$.
\end{lemma}
\begin{proof}
  It suffices to establish the result for $m=3$.
  First of all, since $m = 2 < 3$, Lemma \ref{lem:autoregressive_immersion} yields a residual $U$ such that if $f \in U$, then
  \begin{equation*}
    \mathcal{A}_3(f,\id_X)|_{\{x_1\} \times \dotsb \times \{x_{n-1}\} \times X} : X \to X^m
  \end{equation*}
  is a smooth immersion.  

  We now handle injectivity.
  For convenience, define $F: X^3 \to X^3$ by
  \begin{equation*}
    F(x,y,z) := f(x_1, \dotsc, x_{n-3},x, y, z).
  \end{equation*}
  With this definition, we have
  \begin{equation*}
    \left(\mathcal{A}_3(f,\id_X)\right)(z) = \left(\begin{aligned}&F(x_{n-2},x_{n-1},z),\\ &F(x_{n-1},z, F(x_{n-2},x_{n-1},z)),\\ &F(z,F(x_{n-2},x_{n-1},z),F(x_{n-1},x, F(x_{n-2},x_{n-1},z)))\end{aligned}\right).
  \end{equation*}
  Mirroring the above formula, define $\Phi: X^3 \to X^3$ by
  \begin{equation*}
    \Phi(x,y,z) := \left(\begin{aligned}&F(x,y,z),\\ &F(y,z, F(x,y,z)),\\ &F(z,F(x,y,z),F(x,y, F(x,y,z)))\end{aligned}\right).
  \end{equation*}
  
  The Lemma is established if $\Phi$ is injective upon restricting it to its last coordinate, namely $z$, and taking $x = x_{n-2}$ and $y= x_{n-1}$ as specified in the hypothesis.
  To that end, we follow the plan laid out in \cite[Thm 5.7]{Golubitsky_1973}.
  Define $\Delta (X^3)$ to be the set of tuples of the form $(x,y,z,x,y,z) \in X^6$, and let
  \begin{equation*}
    \Phi_2 : X^3 \times X^3 \backslash \Delta(X^3) \to X^3 \times X^3 \times X^3 \times X^3
  \end{equation*}
  be given by
  \begin{equation*}
    \Phi_2(x,y,z,x',y',z') := (x,y,z,x',y',z',\Phi(x,y,z),\Phi(x',y',z')).
  \end{equation*}
  
  $F$ being injective on its last coordinate (and therefore also $\mathcal{A}_3(f,\id_X)$) is equivalent to the statement that the image $\Phi_2(x,y,X)$ does not intersect the submanifold $W= (\{x\}\times \{y\} \times X)^2 \times \Delta(X^3)$, again for the $x = x_{n-2}$ and $y= x_{n-1}$ specified in the hypothesis.

  According to \cite[Thm 4.13]{Golubitsky_1973}, there is a residual subset $U' \subseteq C^\infty(X^3,X^3)$ such that if $\Phi \in U'$, then $\Phi_2$ is transverse to $W$.
  The submanifold $W$ is of codimension
  \begin{equation*}
    2 (\dim X^2 ) + \dim X^3 = 7 \dim X,
  \end{equation*}
  so the codimension of $\Phi_2^{-1}(W) \subset X^6$ is also $7 \dim X$.
  Since this codimension is strictly greater than the dimension of the domain of $\Phi_2$, namely $6 \dim X$, by \cite[Lemma 5.1]{Golubitsky_1973} transversality is equivalent to not intersecting $W$.

  To complete the argument, recognize that the residual subset $U' \subseteq C^\infty(X^3,X^3)$ restricts to a residual subset $U'' \subseteq C^\infty(X^3,X)$ by retaining only the first coordinate of the codomain.  In this way, if $\Phi \in U'$, we obtain $F \in U''$, and hence $f$ with the desired property.  We therefore define $V := U \cap U''$, to obtain the residual set of $f$ such that $\mathcal{A}_3(f,\id_X)$ is an injective immersion.
\end{proof}

Finally, we can address the proof of the main theorem.

\begin{proof}(of Theorem \ref{thm:autoregressive_embedding})
  Without loss of generality, we assume that $\dim Y \le \dim X$.
  This means that
  \begin{equation*}
    \rank dg \le \dim Y.
  \end{equation*}
  Notice that if $\dim Y > \dim X$, then the maximum of rank $dg$ is $\dim X$,
  which then requires edits to the proof below \emph{mutatis mutandis}.
  
  The main point, as in Lemma \ref{lem:autoregressive_immersion} is a calculation of the derivative of $\mathcal{A}_m(f,\id_X)$ in the last coordinate.
  Observe that this only adds more (admittedly complicated) blocks to our matrices.
  After some work, we find that
  \begin{equation*}
    d \left(\mathcal{A}_m(f,\id_X)|_{\{x_1\} \times \dotsb \times \{x_{n-1}\} \times X}\right)(x_n) = M \begin{pmatrix} \frac{\partial f}{\partial x_n}(x) \\ \frac{\partial f}{\partial x_{n-1}}((\sigma f)(x)) \\ \vdots \\ \frac{\partial f}{\partial x_{n-m+1}}((\sigma f)^{\circ (m-1)}(x))  \end{pmatrix},
  \end{equation*}
  where $M$ is a full rank matrix.
  From this, it follows that
  \begin{equation*}
    \begin{aligned}
      d (\mathcal{A}_m(f,g)&|_{\{x_1\} \times \dotsb \times \{x_{n-1}\} \times X})(x_n) = \\
      &\left(\begin{smallmatrix}
    d_{f(x)} g \\
    & d_{(f\circ (\sigma f))(x)} g \\
    && \ddots \\
    &&& d_{(f \circ (\sigma f)^{\circ m-1}(x)}g\\
    \end{smallmatrix}\right) M \left(\begin{smallmatrix} \frac{\partial f}{\partial x_n}(x) \\ \frac{\partial f}{\partial x_{n-1}}((\sigma f)(x)) \\ \vdots \\ \frac{\partial f}{\partial x_{n-m+1}}((\sigma f)^{\circ (m-1)}(x))  \end{smallmatrix}\right),
    \end{aligned}
  \end{equation*}
  which is evidently full rank provided each of the diagonal $d_*g$ blocks are full rank.
  This is ensured by the classical Sard's theorem on a residual subset of functions in $C^\infty(X,Y)$.
  We can more usefully express this fact by recognizing that each $d_*g$ block must be of rank $\min\{\dim X,\dim Y\}$,
  which yields a manifold of codimension $0$ within the $1$-jet bundle $J^1(X,Y)$.
  The Thom transvserality theorem \cite[Thm 4.9]{Golubitsky_1973} asserts that on a residual subset of $V'$ of $C^\infty(X,Y)$, if $g \in V'$ the map taking $x$ to $(x,d_xg)$ (namely the $1$-jet $j^1g$) will be transverse to the submanifold of codimension $0$ expressing that each block is full rank.
  Therefore the preimage of this submanifold will also have codimension $0$ in $X^m$.
  Briefly, $d_*g$ will be of full rank except at isolated points for $g$ in $V'$.
  Let $K$ the $0$-dimensional submanifold containing these points,
  which we will handle later.
  
  Now to establish injectivity. 
  As in Lemma \ref{lem:autoregressive_embedding_3}, we follow the proof of \cite[Thm 5.7]{Golubitsky_1973}.  
  Multijet transversality \cite[Thm 4.13]{Golubitsky_1973} states that for a residual set $V$ of $C^\infty(X,Y)$, if $g \in V$, then the function $G_2$ defined by
  \begin{equation*}
    G_2(u_1, \dotsc, u_m, v_1 \dotsc, v_m) := \left(u_1, v_1, g(u_1), g(v_1), \dotsc, u_m, v_m, g(u_m), g(v_m)\right)
  \end{equation*}
  is transverse to a submanifold $L := (Y^3)^m$, that identifies self-intersections.
  $L$ is of codimension $m \dim Y$.
  Therefore, the codimension of the preimage of $L$ is $m \dim Y$.
  (Note that the codimension of $L$ is not greater than the dimension of the domain of $G_2$, namely $m \dim X$, so we do not expect injectivity of $g \times \dotsb \times g$.)

  Notice that the hypotheses automatically ensure that $m \ge 3$, so Lemma \ref{lem:autoregressive_embedding_3} applies to give the residual $U'$ such that $f \in U'$ then
  \begin{equation*}
    \mathcal{A}_m(f,\id_X)|_{\{x_1\} \times \dotsb \times \{x_{n-1}\} \times X} : X \to X^m
  \end{equation*}
  is an injective immersion.
   On the other hand, in the proof of Lemma \ref{lem:autoregressive_embedding_3}, the function $\Phi_2$ is transverse to the preimage $G_2^{-1}(L)$ for $f$ in a residual subset $U'' \subseteq U'$ of $C^\infty(X^n,X)$.
     Finally, since the image of $\mathcal{A}_m(f,\id_X)_{\{x_1\} \times \dotsb \times \{x_{n-1}\} \times X}$ is of dimension $d$, hence codimension $(m-1)d$, this will not intersect any of the isolated points $K$ for any $f$ in a (somewhat smaller) residual subset $U \subset U'' \subset C^\infty(X^n,X)$.
  Therefore, for such a choice of $f$, $\mathcal{A}_m(f,g)_{\{x_1\} \times \dotsb \times \{x_{n-1}\} \times X}$ is an immersion.
 
  Therefore, the preimage $\Phi_2^{-1}(G_2^{-1}(L))$ is of codimension $m \dim Y$ in $X^2$.  According to the hypotheses of the Theorem, we assume that
  \begin{equation*}
    2 \dim Z = 2d < m \ell = m \dim Y,
  \end{equation*}
  so by \cite[Lemma 5.1]{Golubitsky_1973}, the composition that defines $\mathcal{A}_m(f,g)$ is an injective map when restricted to the submanifold $Z$.
\end{proof}

\end{document}